\documentclass{lms-b}
\usepackage{amssymb,amsmath,citesort,setspace}
\usepackage{pdfsync}

\newtheorem{theorem}{Theorem}[section]

\newtheorem{lemma}[theorem]{Lemma}
\newtheorem{corollary}[theorem]{Corollary}

\newnumbered{definition}[theorem]{Definition}
\newnumbered{example}[theorem]{Example}
\newnumbered{problem}{Open problem}
\newnumbered{remark}[theorem]{Remark}
\newcommand{\dimh}{\operatorname{\dim_{_{\rm H}}}}
\newcommand{\dimM}{\operatorname{\overline{\dim}_{_{\rm M}}}}

\newcommand{\dimp}{\operatorname{\dim_{_{\rm P}}}}
\newcommand{\bDim}{\operatorname{\overline{\textup{B-Dim}}}}
\newcommand{\BDim}{\operatorname{\overline{\textup{Dim}}}}
\newcommand{\Dim}{\operatorname{Dim}}
\newcommand{\1}{\mathbf{1}}
\newcommand{\Pp}{\mathrm{P}}

\newcommand{\E}{\mathrm{E}}
\newcommand{\R}{\mathbf{R}}
\newcommand{\e}{\epsilon}

\renewcommand{\d}{\textup{d}}
\title[Packing Dimension Profiles and L\'evy Processes]{Packing Dimension Profiles and L\'evy Processes}
\author{D.\ Khoshnevisan, R.L.\ Schilling and Y.\ Xiao}
\classno{60G17 (primary), 60J25, 60G51, 28A80}
\extraline{Research supported in part by the grant DMS-1006903 from the National Science Foundation}
\begin{document}
\maketitle
\begin{abstract}
    We extend the concept of packing dimension profiles, due to Falconer
    and Howroyd (1997) and Howroyd (2001), and use our extension
    in order to determine the packing dimension of an arbitrary image of a general L\'evy process.
\end{abstract}

\maketitle
\section{Introduction}

Let $X:= \{X(t)\}_{t\ge 0}$ be a L\'evy process in $\R^d$; that is,
$X(0)=0$, $X$ has stationary and independent increments, and
the random function $t\mapsto X(t)$ is almost surely right continuous with left limits
\cite{Applebaum,Bertoin96,Fristedt,Kyprianou,Sato}.

Let $F$ be a nonrandom Borel subset of $\R_+:=[0\,,\infty)$.
It has been known for a long time that the random image set $X(F)$
frequently exhibits fractal structure. And there is a substantial literature
that computes the Hausdorff dimension $\dimh X(F)$ of $X(F)$, see \cite{KX:05} and
its extensive bibliography.

Let $\dimp$ denote the packing dimension.
The main goal of  the present paper is to evaluate $\dimp X(F)$
in terms of  the geometry of $F$. In order to accomplish
this, we shall introduce and study a new family of dimensions related to the set
$F$. Those dimensions are inherently probabilistic, but they have
analytic significance as well. In fact, they can be viewed as an extension of
the notion of packing dimension profiles, as introduced by Falconer and
Howroyd \cite{FalconerHowroyd} to study the packing dimension of
orthogonal projections; see also Howroyd \cite{Howroyd}.

There exists an extensive body of literature related to the Hausdorff dimension
$\dimh X(F)$, but only few papers study the packing dimension
$\dimp X(F)$. Let us point out two noteworthy cases where
$\dimp X(F)$ has been computed successfully in different settings:
\begin{description}
\item[Case 1.]
    When $\dimh F = \dimp F$,
    covering arguments can frequently be used to compute $\dimp X(F)$.
    In those cases, the packing and Hausdorff dimensions of $X(F)$
    generally agree.
\item[Case 2.] When $X$ has statistical self similarities, one can sometimes
    appeal to scaling arguments in order to compute
    $\dimp X(F)$ solely in terms of $\dimp F$.
\end{description}

For an example of the more interesting Case 2
consider the situation where $X$ is an isotropic stable
process on $\R^d$ with index $\alpha\in(0\,,2]$; that is the case where
$\E\exp(i\xi\cdot X(t))=\exp(-\text{const} \cdot t\|\xi\|^\alpha)$
simultaneously for all $t>0$
and $\xi\in\R^d$. Then a theorem
of Perkins and Taylor \cite{PT87} implies that
if $d\ge\alpha$, then
\begin{equation}\label{eq:PT:87}
    \dimp X(F)=\alpha\dimp F\qquad\text{a.s.}
\end{equation}
For related works see also \cite{MeerschaertXiao,XiaoLiu1}.
Up to now, the case $d<\alpha$ has remained open, except
when $\alpha=2$ and $d=1$ [i.e., $X=$ linear Brownian motion].
In that case, a more general theorem of Xiao \cite{Xiao:97} implies
that $\dimp X(F)$ cannot, in general, be described by $\dimp F$;
in fact, Xiao's theorem implies that, when $X$ is linear Brownian motion,
\begin{equation}\label{eq:Xiao:97}
    \dimp X(F)=2\Dim^{\textrm{FH}}_{1/2} F
    \qquad\text{a.s,}
\end{equation}
where $\Dim^{\textrm{FH}}_sF$ denotes the $s$-dimensional packing dimension
profile of Falconer and Howroyd \cite{FalconerHowroyd}. The complexity of
the preceding formula can be appreciated better in light of an example
of Talagrand and Xiao \cite{TX} which shows that there are
sets $F$ such that: (i) $\dimp X(F)\neq \dimh X(F)$; and (ii)
$\dimp X(F)$ cannot be described solely in terms of simple-to-understand
quantities such as $\dimh F$ and $\dimp F$.

The main goal of this paper is to introduce a new family of
dimension profiles; this family includes the packing dimension profiles
of Falconer and Howroyd \cite{FalconerHowroyd}. We use the
dimension profiles of this paper to compute
$\dimp X(F)$ for a general L\'evy process $X$ and an
arbitrary nonrandom
Borel set $F\subset\R_+$. En route we also establish a novel
formula for $\dimM X(F)$, where $\dimM$ denotes the
upper Minkowski dimension.

In order to understand our forthcoming identities better, let us mention
three corollaries of the general results of this paper.

\begin{corollary}\label{co:stable}
    Let $X$ denote a one-dimensional isotropic stable L\'evy process
    with index $\alpha\in(0\,,2]$ and $F\subseteq\R_+$ be
    nonrandom and Borel measurable. Then,
    \begin{equation}
        \dimp X(F) =\alpha\, \Dim^{\textrm{FH}}_{d/\alpha} F
        \qquad\text{a.s.}
    \end{equation}
\end{corollary}

The preceding includes both \eqref{eq:PT:87} and
\eqref{eq:Xiao:97} as special cases. Indeed, one
obtains \eqref{eq:PT:87}  because $\Dim_s^{\textrm{FH}}F=\dimp F$
when $s\ge 1$; see Falconer and Howroyd \cite{FalconerHowroyd}.
And one obtains \eqref{eq:Xiao:97} by setting $d=1$ and $\alpha=2$.
We mention that even in the preceding setting, the extension from
$\alpha=2$ to $\alpha<2$ is not trivial, since in the latter case
$t\mapsto X(t)$ is pure jump. And this will force us to develop new ideas
to handle pure jump processes,  even when
$X$ is $\alpha$-stable.

In order to describe our next two corollaries, let us recall that a
stochastic process $S:=\{S(t)\}_{t\ge 0}$ is a \emph{subordinator}
if $S$ is a one-dimensional L\'evy process such that
the random function $t\mapsto S(t)$
is nondecreasing. Also recall that the \emph{Laplace exponent} $\Phi$
is $S$ is given by $\E{\rm e}^{-\lambda S(t)}={\rm e}^{-t\Phi(\lambda)}$
for every $t,\lambda>0$; see Bertoin \cite{bertoin:99} for more detailed
information about subordinators and their remarkable properties.

\begin{corollary}\label{co:HA:subordinator}
    Suppose that $S$ is a subordinator with Laplace exponent $\Phi$,
    and $F\subseteq\R_+$ is nonrandom and Borel measurable.
    Then, a.s.,
    \begin{equation}\label{eq:HA:subordinator}
    	\dimM S(F)\\
	=\sup\left\{\eta>0:\
	        \varlimsup_{\lambda\uparrow\infty} \lambda^\eta
	        \inf_{\nu\in\mathcal{P}_c(F)} \int\int
	        {\rm e}^{-|t-s|\Phi(\lambda)}\, \nu(\d s)\,\nu(\d t)=0
	        \right\},
    \end{equation}
    where $\mathcal{P}_c(F)$ denotes the collection of all
    compactly supported Borel probability measures $\nu$ such that
    $\nu(F)=1$.
\end{corollary}

Consider Corollary \ref{co:HA:subordinator} in the case that $F$ is
an interval; say $F:=[0\,,1]$. Then it is intuitively plausible---and
possible to prove rigorously---that the minimizing measure $\nu$
in the infiumum ``$\inf_{\nu\in\mathcal{P}_c(F)}$''
is the Lebesgue measure on $[0\,,1]$. A direct calculation then
implies that the convergence condition
of \eqref{eq:HA:subordinator} holds if, and only if,
$\lambda^\eta=o(\Phi(\lambda))$ as $\lambda\uparrow\infty$.
Therefore, Corollary \ref{co:HA:subordinator} yields the following
elegant a.s.\ identity:
\begin{equation}\label{eq:dimM:subord}
    \dimM S([0\,,1])=\sup\left\{\eta>0:\
    \varliminf_{\lambda\uparrow\infty}\frac{\Phi(\lambda)}{\lambda^\eta}=\infty\right\}
    \qquad\text{a.s.}
\end{equation}
And a Baire-category argument can be used to prove that the same
formula holds if we replace $\dimM S([0\,,1])$ by $\dimp S([0\,,1])$;
see also \eqref{eq:dimp} below.
The preceding formulas for $\dimM S([0\,,1])$
and $\dimp S([0\,,1])$ were derived earlier, using covering
arguments;
see Fristedt and Taylor \cite{FristedtTaylor}
and Bertoin \cite[Lemma 5.2, p.\ 41]{bertoin:99}.

We are not aware of any nontrivial
examples of deterministic sets with explicitly known Falconer--Howroyd
packing dimension profiles. Remarkably,
our third and final corollary computes the Falconer--Howroyd
packing dimension profiles of a quite-general
``Markov random set'' in the sense of Krylov and Juskevi\v{c}
\cite{KrylovJuskevic64,KrylovJuskevic65}; see also
Hoffmann--J\o{}rgensen \cite{HoffmannJorgensen} and
Kingman \cite{Kingman}. According to a deep result of
Maisonneuve \cite{Maisonneuve}, a Markov random set is the closure of
$S(\R_+)$, where $S$ is a subordinator. It is not hard to see that
any reasonable dimension of the closure of
$S(\R_+)$ is a.s.\ the same as the
same dimension of $S([0\,,1])$. Therefore, our next
corollary concentrates on computing
many of the Falconer--Howroyd packing dimension profiles of $S([0\,,1])$.

\begin{corollary}\label{co:Dim:Sub}
    If $S$ is a subordinator with Laplace exponent $\Phi$, then for
    every $s\ge 1/2$,
    \begin{equation}\label{eq:Dim:Sub}
        \Dim^{\rm FH}_s  S([0\,,1]) = s(1-\theta)
        \qquad\text{a.s.,}
    \end{equation}
    where
    \begin{equation}
        \theta := \varliminf_{\lambda\uparrow\infty}\frac{1}{\log\lambda}
        \log\left( \int_1^\lambda \frac{\d x}{\Phi(x^{1/s})} \right).
    \end{equation}
\end{corollary}

\begin{remark}
    The preceding should be compared with the following:
    With probability one:
    \begin{equation}\begin{split}
        \dimp S([0\,,1]) &= \varlimsup_{\lambda\uparrow\infty}
            \frac{\log\Phi(\lambda)}{\log\lambda};\qquad\text{and}\\[\bigskipamount]
        \dimh S([0\,,1]) &= \varliminf_{\lambda\uparrow\infty}
            \frac{\log\Phi(\lambda)}{\log\lambda}.
    \end{split}\end{equation}
    See, for example, Fristedt and Taylor \cite{FristedtTaylor},
    as well as Bertoin \cite[Lemma 5.2, p.\ 41, Corollary 5.3, p.\ 42]{bertoin:99}.
    See also \eqref{eq:dimM:subord} above and see \cite{KX:07b} for
    very general results.
\end{remark}

\section{Analytic preliminaries and the main result}

In this section we introduce a family of generalized
packing dimension profiles that are associated to a L\'evy process
in a natural way. As we shall see later, these profiles include
the packing dimension profiles of Falconer and Howroyd
\cite{FalconerHowroyd} and Howroyd \cite{Howroyd}.
We mention also that it has been
shown in \cite{KX:07a} that the packing dimension profiles of
Falconer and Howroyd \cite{FalconerHowroyd} and those of Howroyd \cite{Howroyd}
coincide. See also Howroyd \cite{Howroyd} for a special case of
the latter result.

\subsection{Packing dimension profiles}
Recall that $X:=\{X(t)\}_{t\ge 0}$ is an arbitrary but fixed L\'evy process on $\R^d$.
If $|y|:=\max_{1\le j\le d}|y_j|$ designates the $\ell^\infty$ norm of
a vector $y\in\R^d$, then we may consider the family
\begin{equation}\label{eq:kappa}
    \kappa :=\{\kappa_\e\}_{\e\ge 0}
\end{equation}
of functions that are defined by
\begin{equation}\label{eq:kappa2}
    \kappa_\e(t) := \Pp\left\{ X(t)\in B(0\,,\e)\right\}
    \qquad\text{for all  $\e,t\ge 0$.}
\end{equation}
Here and throughout $B(x\,,r):=\{z\in\R^d:\ |z-x|<r\}$ denotes the
open $\ell^\infty$ ball of radius $r>0$ about $x\in\R^d$.

One can see at once that $\kappa_\e(t)$ is continuous in $t$ for every fixed $\e$,
and nondecreasing in $\e$ for every fixed $t$.

\begin{definition}\label{def:BD}
    We define the \emph{box-dimension profile} $\BDim_\kappa F$ of a Borel set
    $F\subseteq\R$ with respect to the family $\kappa$ as follows:
    \begin{equation} \label{Eq:BD}
        \BDim_\kappa F  := \sup
        \left\{ \eta>0:\,  \varliminf_{\e\downarrow 0}
        \inf_{\nu\in\mathcal{P}(F)}
        \int\int\frac{\kappa_\e(|t-s|)}{\e^\eta} \,
        \nu(\d s)\,\nu(\d t)
        =0 \right\},
    \end{equation}
    where $\nu\in\mathcal{P}(F)$ if, and only if, $\nu$ is a Borel
    probability measure on $\R$ such that  $\nu(F) = 1$.
\end{definition}

It is possible to express $\BDim_\kappa F$
in potential-theoretic terms. Indeed,
\begin{equation} \label{Eq:BD0}
    \BDim_\kappa F  = \varlimsup_{\e \to 0} \frac{\log Z_\kappa(\e)}
    {\log \e},
\end{equation}
where $Z_\kappa(\e)$ is the minimum $\kappa_\epsilon$-energy of $\nu \in \mathcal{P}(F)$;
viz.,
\begin{equation}
    Z_\kappa(\e) := \inf_{\nu\in\mathcal{P}(F)} \int\int\kappa_\e(|t-s|)\,
    \nu(\d s)\,\nu(\d t).
\end{equation}
Eq.\ \eqref{Eq:BD0} is reminiscent of, but not the same as, Howroyd's
upper box-dimension with respect to a kernel \cite{Howroyd}.

The box-dimension profiles $\BDim_\kappa$ can be regularized in
order to produce a proper family of packing-type dimensions.

\begin{definition}\label{def:Dimka}
    We define the \emph{packing dimension profile $\Dim_\kappa F$
    of an arbitrary set $F\subseteq\R$ with respect to
    the family $\kappa$} as follows:
    \begin{equation}\label{Def:Dimka}
        \Dim_\kappa F := \inf
        \sup_{n\ge 1} \BDim_\kappa  F_n,
    \end{equation}
    where the infimum is taken over all countable coverings
    of $F$ by bounded Borel sets
    $F_1,F_2,\ldots$
\end{definition}

One can verify from this definition that $\Dim_\kappa$ has the
following properties that are expected to hold for any reasonable
notion of ``fractal dimension'':
\begin{itemize}
\item[(i)]\ $\Dim_\kappa$ is \emph{monotone}. Namely,
    $\Dim_\kappa F \le  \Dim_\kappa G$ whenever
    $F \subseteq G$;
\item[(ii)]\ $\Dim_\kappa$ is \emph{$\sigma$-stable}. Namely,
    $\Dim_\kappa  \cup_{n=1}^\infty G_n  =
    \sup_{n \ge 1} \Dim_\kappa G_n.$
\end{itemize}
We skip the verification of these properties, as they require routine arguments.

\subsection{A relation to a family of packing measures}
Next we outline how $\Dim_\kappa$ can be associated to a
packing measure with respect to the family $\kappa$,
where $\kappa$ was defined in \eqref{eq:kappa}
and \eqref{eq:kappa2}.

\begin{definition}
    Fix a set $F \subseteq \R$ and a number $\delta>0$.
    We say that a sequence $\{(w_j\,, t_j\,, \e_j):\, j \ge 1\}$ of triplets is a
    \emph{$(\kappa\,, \delta)$-packing of $F$} if for all $j\ge 1$:
    (a) $w_j \ge 0$;
    (b) $t_j \in F$;
    (c) $\e_j \in (0\,, \delta)$; and
    (d) $\sum_{i=1}^\infty w_i \kappa_{\e_i}(|t_i - t_j|) \le 1$.
\end{definition}

The preceding general definition is modeled after the ideas of
\cite{Howroyd}, and leads readily to packing measures. Indeed, we
have the following.

\begin{definition}\label{def:Pmeaska}
    For a given constant $s > 0$, we define the \emph{$s$-dimensional
    packing measure ${\mathcal P}^{s, \kappa}(F)$ of
    $F\subset\R$ with respect to the family $\kappa$} as
    \begin{equation}
        {\mathcal P}^{s, \kappa}(F) := \inf
        \sup_{n\ge 1}   {\mathcal P}^{s, \kappa}_0 (F_n),
    \end{equation}
    where the infimum is taken over all bounded Borel sets $F_1,F_2,\ldots$
    such that $F\subseteq\cup_{n=1}^\infty F_n$,
    and ${\mathcal P}^{s, \kappa}_0$ denotes a so-called \emph{premeasure}
    that is defined by 
    \begin{equation}
        {\mathcal P}^{s, \kappa}_0(F) := \lim_{\delta \downarrow 0}\left(
        \sup \sum_{j= 1}^\infty w_j \e_j^s\right),
    \end{equation}
    where the supremum is taken over all $(\kappa\,,\delta)$-packings
    $\{(w_j\,,t_j\,,\e_j):\, j\ge 1\}$ of $F$ and  $\sup \varnothing := 0$,
    as usual.
    \end{definition}

\begin{definition}\label{def:Pmeaska2}
    We define the \emph{packing dimension
    $\hbox{\rm P-}{\dim}_{\kappa} F$ of $F\subset\R$
    with respect to the family $\kappa$}
    as $\hbox{\rm P-}{\dim}_{\kappa} F :=
    \inf \{s > 0:\, {\mathcal P}^{s,\kappa}(F) =0\}.$
\end{definition}

It is possible to adapt the proof of Theorem 26 of
Howroyd's paper \cite{Howroyd} and show that our two packing
dimension profiles coincide. That is,
\begin{equation}
    \Dim_\kappa F=\hbox{\rm P-}{\dim}_{\kappa} F
    \qquad\text{for all $F\subset\R$}.
\end{equation}
We omit the proof, as it requires only an adaptation of ideas of Howroyd
\cite[proof of Theorem 26]{Howroyd} to the present,
more general, setting.

\subsection{A relation to harmonic analysis}

There are time-honored, as well as deep, connections between Hausdorff
measures and harmonic analysis. In this section we establish a useful
harmonic-analytic result about the packing measures of this
section.

Let $\Psi$ denote the characteristic exponent of $X$, normalized so that
\begin{equation}
    \E{\rm e}^{iz\cdot X(t)} = {\rm e}^{-t\Psi(z)}
    \qquad\text{for all $t\ge 0$ and $z\in\R^d$.}
\end{equation}
For every Borel probability measure $\nu$ on $\R$, and for all $z\in\R$,
define the energy form,
\begin{equation}\label{eq:E_rho}
    \mathcal{E}_\nu(z) := \int\int \exp\left(
    -|t-s|\Psi \left( \textrm{sgn}(t-s) z\right) \right)\, \nu(\d s)\,\nu(\d t).
\end{equation}

Note that
\begin{equation}
	0\le \mathcal{E}_\nu(z)\le 1\qquad
	\text{for all $z\in\R^d$ and $\nu\in\mathcal{P}(F)$.}
\end{equation}
This can be seen from the following computation
\begin{equation}\begin{split}
	\mathcal{E}_\nu(z) &=\int\int \E\left[ {\rm e}^{%
		iz\cdot\{ X(t)-X(s)\}}
		\right]\, \nu(\d s)\,\nu(\d t)\\
	&= \E\left( \left|\int {\rm e}^{iz\cdot X(t)}\, \nu(\d t)\right|^2 \right),
\end{split}\end{equation}
where we used Fubini's theorem to interchange the order of the integrals.
This proves that $\mathcal E_\nu$ is real-valued and positive;
the fact that $\mathcal E_\nu(z)\leq 1$ is now obvious.

\begin{theorem}\label{th:HA}
    For every compact set $F\subset\R_+$,
    \begin{equation}\label{eq:th:HA}
        \BDim_\kappa F  = \sup\left\{\eta>0:\
        \varliminf_{\e\downarrow 0}
        \inf_{\nu\in\mathcal{P}(F)}
        \int_{\R^d} \frac{\e^{-\eta}\,\mathcal{E}_\nu(z/\e)}{\prod_{j=1}^d
        (1+z_j^2)}\, \d z=0 \right\}.
    \end{equation}
\end{theorem}

\begin{proof}
    We apply a variation of  the \emph{Cauchy semigroup argument} of
    \cite[proof of Theorem 1.1]{KX:07b}.
    Define for all $\e>0$, the (scaled) P\'olya distribution,
    \begin{equation}
        P_\e(x) := \prod_{j=1}^d \left( \frac{1-\cos(2\e x_j)}{
        2\pi\e x_j^2}\right)\qquad\text{for all $x\in\R^d.$}
    \end{equation}
    Then it is well-known, as well as elementary, that
    \begin{equation}
        \hat{P}_\e (\xi) = \prod_{j=1}^d\left(
        1 - \frac{|\xi_j|}{2\e} \right)^+\qquad\text{for all $\xi\in\R$},
    \end{equation}
    where $a^+:=\max(a\,,0)$ for all real numbers $a$, and
    $\hat{f}$ denotes the Fourier transform of $f$ normalized so that
    $\hat{f} (z) =\int_{\R^d} f(x)\exp(ix\cdot z)\, \d x$
    for all integrable functions $f:\R^d\to\R$.
    If $z\in B(0\,,\e)$, then $1-(2\e)^{-1} |z_j|\ge \frac12$.
    Consequently,
    $\1_{B(0,\e)}(z)\le 2^d \hat{P}_\e (z)$ for all
    $z\in\R^d$. Set $z:=X(t)$ and take expectations in the preceding
    inequality to find that for all $\e>0$ and $t\ge 0$,
    \begin{equation}
        \kappa_\e(t)
        \le 2^d \, \E \left[ \hat{P}_\e (X(t)) \right]
        = 2^d \int_{\R^d} \hat P_\e(z)\, \Pp(X(t)\in dz).
    \end{equation}
    We can apply Plancherel's identity
    to the right-hand side of
     this inequality and deduce that
    \begin{equation}
        \kappa_\e(t) \le 2^d\int_{\R^d} P_\e(y) {\rm e}^{-t\Psi(y)}
        \,\d y\qquad\text{for all $\e>0,\, t\ge 0.$}
    \end{equation}
    On the other hand, if $t<0$ then we use the L\'evy process
    $-X(-t)$ in place of $X(t)$ to deduce that
    $\kappa_\e(-t) \le 2^d \, \E[ \hat{P}_\e (-X(-t))]$,
    whence it follows that
    \begin{equation}
        \kappa_\e(-t) \le 2^d\int_{\R^d} P_\e(y) {\rm e}^{t\Psi(-y)}
        \,\d y\qquad\text{for all $\e>0$ and  $t< 0.$}
    \end{equation}
    Consequently, the following holds for all $\e>0$ and $t\in\R$:
    \begin{equation}\label{Eq:ka}
        \kappa_\e(|t|) \le 2^d\int_{\R^d} P_\e(y)
        \exp\left( -|t|\Psi\left( \textrm{sgn}(t)y\right) \right)
        \,\d y.
    \end{equation}

    Define $f_C$ to be the standard Cauchy density on $\R^d$;
    that is,
    \begin{equation}
        f_C(z):= \pi^{-d} \prod_{j=1}^d\left(1+z_j^2\right)^{-1}
        \qquad\text{for every $z:=(z_1\,,\ldots,z_d)\in\R^d$}.
    \end{equation}
	Because of the elementary inequality
	\begin{equation}
		\frac{1-\cos(2u)}{2 \pi\, u^2}
		= \frac{\sin^2 u}{\pi\, u^2}
		\leq \frac 1{1+u^2},
    \end{equation}
    valid for all nonzero $u$,
    it follows that $P_\e(y) \le (\pi\e)^d f_C(y)$ for all $y\in\R^d$.
    Thus, (\ref{Eq:ka}) and a change of variables imply that
    \begin{equation}\label{bd1}
        \int\int \kappa_\e(|t-s|)\,\nu(\d t)\,\nu(\d s)
        \le (2\pi)^d \int_{\R^d} f_C(z)\, \mathcal{E}_\nu
        \left(\frac{z}{\e}\right)\, \d z.
    \end{equation}
	This shows that every $\eta$   that  is smaller than the
	right-hand side of \eqref{eq:th:HA} also satisfies
	$\eta < \BDim_\kappa F$. Consequently
	$\BDim_\kappa F$ is larger or equal than the
	supremum that appears in \eqref{eq:th:HA}.
	
    Let us now establish the converse estimate.
    After enlarging the underlying probability space if
    necessary, we can introduce a Cauchy process
    $C:=\{C(t)\}_{t\ge 0}$---independent of
    $X$---whose coordinate processes $C_1,\ldots,C_d$ are i.i.d.\
    standard symmetric Cauchy processes on the line.
     For every  $\e>0$, $k\geq 1$, and $x\in\R^d$,
    \begin{equation}\label{eq:new}\begin{split}
        \E\exp\left[ix\cdot C( k/\e)\right] - {\rm e}^{-k}
        &= {\rm e}^{-k |x|_1/\e} - {\rm e}^{-k}\\
        &\le {\rm e}^{-k}\left({\rm e}^{-k\left(|x|-\e\right)/\e} - 1\right)\\
        &\leq \1_{B(0,\e)}(x).
    \end{split}\end{equation}
    In the above, $|x|_1 = \sum_{j=1}^d |x_j|$ is the $L^1$ norm of $x\in \R^d$.

   If $t\ge 0$, then we set $x:= X(t)$ in \eqref{eq:new} and take expectations
   to find that
    \begin{equation}\begin{split}
        \kappa_\e(t) &\ge \E\left[ \exp\left( i X(t) \cdot C(k/\e) \right) \right]
            -{\rm e}^{-k}\\
        &= \E\exp\left( - t \Psi(C(k/\e)) \right) - {\rm e}^{-k}\\
        &= \int_{\R^d} f_C(z) {\rm e}^{-t\Psi(kz/\e)}\, \d z - {\rm e}^{-k}.
    \end{split}\end{equation}
    If $t<0$, a similar calculation with $x:=-X(-t)$ in place of $X(t)$ yields
    \begin{equation}\begin{split}
        \kappa_\e(-t) &\ge \E\left[ \exp\left( -i X(-t) \cdot C(k/\e) \right) \right]
            -{\rm e}^{-k}\\
        &= \int_{\R^d} f_C(z) {\rm e}^{t\Psi(-kz/\e)}\, \d z - {\rm e}^{-k}.
    \end{split}\end{equation}
    Thus, for all $t\in\R$,
    \begin{equation}
        \kappa_\e(|t|)
        \ge \int_{\R^d} f_C(z) \exp\left( -|t|
        \Psi \left( \textrm{sgn}(t) kz/\e \right) \right)\, \d z - {\rm e}^{-k}.
    \end{equation}
    We choose $k:=\e^{-\delta}$, where $\delta$ is positive but arbitrarily
    small, replace $|t|$ by $|t-s|$, and integrate with respect to
    $\nu(\d t)\,\nu(\d s)$ to find that
    \begin{equation}
        \int\int \kappa_\e(|t-s|)\,\nu(\d s)\,\nu(\d t)\\
        \ge \int_{\R^d} f_C(z) \,\mathcal{E}_\nu\left(
        	\frac{z}{\e^{1+\delta}}\right) \, \d z -
		\exp\left( -\left[\frac 1{\e^{1+\delta}}\right]^{1+1/\delta}
		\right).
    \end{equation}
     If $\eta<\BDim_\kappa F$, then
	\begin{equation}
		\varliminf_{\e\downarrow 0} \inf_{\nu\in\mathcal P(F)}
		\int\int \frac{\kappa_\e(|t-s|)}{\e^\eta}\,\nu(\d s)\,\nu(\d t) = 0,
	\end{equation}
    and hence the preceding discussion implies that
    \begin{equation}
    	\varliminf_{h\downarrow 0} \inf_{\nu\in\mathcal P(F)}
	h^{-\eta/(1+\delta)} \int_{\R^d} f_C(z) \,\mathcal{E}_\nu
	(z/h) \, \d z = 0.
	\end{equation}
	That is, $\eta/(1+\delta)$ is smaller than the supremum that appears
	on the right-hand side of \eqref{eq:th:HA}. This implies that the
	right-hand side of \eqref{eq:th:HA} is less or equal than
	$\BDim_\kappa F / (1+\delta)$, whence the theorem,
	because $\delta$ is arbitrary.
\end{proof}

\subsection{The main result and proofs of corollaries}

\begin{theorem}\label{th:dimM}
	Let $X := \{X(t)\}_{t\ge 0}$ denote a L\'evy process in $\R^d$ and let
	$\kappa$ be defined by \eqref{eq:kappa} and
	\eqref{eq:kappa2}.  Then, for all nonrandom bounded Borel sets
	$F\subseteq\R_+$:
	\begin{align}\label{eq:dimM}
		\dimM X(F) &= \BDim_\kappa F \quad a.s.;\quad\text{and}\\
		\dimp X(F) &= \Dim_\kappa F\quad\text{a.s.}
		\label{eq:dimp}
	\end{align}
\end{theorem}
Theorem \ref{th:dimM} is proved in the Section \ref{sec3} below.
In the remaining part of this section we apply Theorem \ref{th:dimM}
in order to verify the three corollaries---Corollary \ref{co:stable},
\ref{co:HA:subordinator} and \ref{co:Dim:Sub}---that were mentioned
earlier in the introduction.

\begin{proof}[of Corollary \ref{co:stable}]
	It is well known that for all $T>0$ there exist constants
	$0 < A_1 \leq A_2 < \infty$ such that
	uniformly for all $t\in[0\,,T]$ and $\e\in(0\,,1)$,
	\begin{equation}
		A_1 \left( \frac{\e}{t^{1/\alpha}} \wedge 1\right)^d \le
		\kappa_\e(t) \le A_2\left( \frac{\e}{t^{1/\alpha}} \wedge 1\right)^d.
	\end{equation}
	It follows from the very definition of
	$\BDim_\kappa$ that $\BDim_\kappa F$ is equal to the
	supremum of all $\eta>0$ such that
	\begin{equation}
		\varliminf_{\e\downarrow 0} \frac{1}{\e^{\eta/\alpha}}
		\inf_{\nu\in\mathcal{P}(F)} \int\int \left[ \left(
		\frac{\e}{|t-s|}\right)^{d/\alpha} \wedge 1\right]\,\nu(\d s)\,\nu(\d t)=0.
	\end{equation}
	That is, in this case,
	\begin{equation}
		\BDim_\kappa F = \alpha\, \bDim_{d/\alpha} F,
	\end{equation}
	where $\bDim_s F$ denotes the $s$-dimensional
	box-dimension profile of Howroyd \cite{Howroyd}.
	Consequently, we can combine \eqref{eq:dimp} together
	with our earlier result \cite[Theorem 4.1]{KX:07a} to deduce that
	$\Dim_\kappa F= \alpha\, \Dim^{\textrm{FH}}_{d/\alpha} F$,
	where $\Dim^{\textrm{FH}}_s F$ denotes the $s$-dimensional
	packing dimension profile of Falconer and Howroyd \cite{FalconerHowroyd}.
	This concludes the proof.
\end{proof}

\begin{proof}[of Corollary \ref{co:HA:subordinator}]
	Recall that $\Phi$ is the Laplace exponent of the subordinator
	$S$ and write $\Psi$ for its characteristic L\'evy exponent;\
	i.e., $\E \,{\rm e}^{i\xi S(t)} =
	{\rm e}^{-t\Psi(\xi)}$.  We may introduce an independent
	real-valued symmetric Cauchy process $X$ and denote respectively
	by $\E_S$ and $\E_X$ the expectations corresponding to $S$ and $X$.
	In this way we find that for all $\lambda \ge 0$,
	\begin{equation}\begin{split}
		{\rm e}^{-t\Phi(\lambda)}
			= \E_S\, {\rm e}^{-\lambda S(t)}
			&= \E_S\E_X\, {\rm e}^{i S(t) X(\lambda)}
       		 	= \E_X \E_S\, {\rm e}^{i X(\lambda) S(t)}\\
        	&= \E_X\, {\rm e}^{-t\Psi(X(\lambda))}
       			= \frac{1}{\pi}\, \int_{-\infty}^\infty
        		\frac{{\rm e}^{-t\Psi(\lambda z)}}{1+z^2}\, \d z.
    \end{split}\end{equation}
    This and a symmetry argument show that for all $s,t,\lambda\ge 0$,
    \begin{equation}
		{\rm e}^{-|t-s|\Phi(\lambda )} = \frac{1}{\pi}\,
		\int_{-\infty}^\infty
		\frac{{\rm e}^{-|t-s|\Psi(\textrm{sgn}(t-s)\lambda z)}}{1+z^2}\, \d z.
    \end{equation}
    Let $\lambda:=1/\e$, and
    integrate both sides with respect to $\nu(\d s)\,\nu(\d t)$,
    for some $\nu\in\mathcal{P}_c(F)$, to find that
    \begin{equation}
	\frac1\pi
        \int_{-\infty}^\infty \frac{\mathcal{E}_\nu(z/\e)}{1+z^2}\, \d z
        = \int\int {\rm e}^{-|t-s|\Phi(1/\e)}\,\nu(\d s)\,\nu(\d t).
    \end{equation}
    Thus, we obtain the corollary immediately from Theorem \ref{th:HA}
    and Theorem \ref{th:dimM}.
\end{proof}

\begin{proof}[of Corollary \ref{co:Dim:Sub}]
    We can write $s=1/\alpha$, where $\alpha\in(0\,,2]$.
    By enlarging the underlying probability space, if need be,
    we introduce an independent, linear, symmetric stable L\'evy process
    $X_\alpha$ with index $\alpha$. The subordinate process
    $X_\alpha\circ S$ is itself a L\'evy process, and its characteristic
    exponent is $z\mapsto \Phi (|z|^\alpha )$ for $z\in\R$.
    According to Theorem 1.1 of \cite{KX:07b},  the following
    holds a.s.:
    \begin{equation}\label{eq:main:int}
        \dimp X_\alpha(S([0\,,1])) = \varlimsup_{r\downarrow 0}\frac{1}{\log r}
        \log \left(\int_0^\infty \frac{\d x}{(1+x^2)(1+\Phi((x/r)^\alpha))}
        \right).
    \end{equation}
    We analyze the integral by splitting it over three regions.
    Without loss of generality we may assume that $0<r<\frac 12$.

    If $x\in(0\,,r)$, then $0 \le \Phi((x/r)^\alpha)\le\Phi(1)$,
    and hence
    \begin{equation}\label{case1}
        \int_0^r \frac{\d x}{(1+x^2)(1+\Phi((x/r)^\alpha))}
        \asymp r\qquad\text{as $r\downarrow 0$},
    \end{equation}
    where ``$f(r)\asymp g(r)$ as $r\downarrow 0$''
    means that $f(r)/g(r)$
    is bounded above and below by constants that do not depend on
    $r$ as $r\downarrow 0$. 

    Similarly,
    \begin{equation}
        \int_r^1 \frac{\d x}{(1+x^2)(1+\Phi((x/r)^\alpha))}
        \asymp r\int_1^{1/r} \frac{\d x}{\Phi(x^\alpha)} := f(r)
        \quad\text{as $r\downarrow 0$},
    \end{equation}
    and
    \begin{equation}
        \int_1^\infty \frac{\d x}{(1+x^2)(1+\Phi((x/r)^\alpha))}
        \asymp \frac 1r\int_{1/r}^\infty \frac{\d x}{x^2\Phi(x^\alpha)}
        :=g(r)
        \quad\text{as $r\downarrow 0$}.
    \end{equation}

    We first observe that $\int_1^{1/r} 1/\Phi(x^\alpha)\,\d x$ is bounded
    away from zero for all $r\in(0\,,1/2)$. This proves that
    $r=O(f(r))$ as $r\downarrow 0$, and hence the integral in
    \eqref{case1} does not contribute to the limit in
    \eqref{eq:main:int}. In addition,
    $f(r) \ge 1/\Phi(r^{-\alpha})$,
    and hence
    \begin{equation}
        g(r) \le \frac{1}{r\Phi(r^{-\alpha})}\int_{1/r}^\infty
        \frac{\d x}{x^2} \le f(r).
    \end{equation}
    Because $\alpha=1/s$,
    the preceding observations together prove that
    \begin{equation}\label{eq:dis1}
        \dimp X_\alpha(S([0\,,1])) = \varlimsup_{r\downarrow 0}
        \frac{\log f(r)}{\log r} = 1-\theta.
        \qquad\text{a.s.}
    \end{equation}
    On the other hand, we can apply \eqref{eq:dimp},
    conditionally on the process $S$, in order to deduce that
    \begin{equation}\label{eq:dis2}
        \dimp X_\alpha(S([0\,,1])) =
        \alpha \Dim^{\textrm{FH}}_{1/\alpha} S([0\,,1])
        \qquad\text{a.s.;}
    \end{equation}
    see also Corollary \ref{co:stable}.
    Corollary \ref{co:Dim:Sub}
    follows upon combining \eqref{eq:dis1} and \eqref{eq:dis2}.
\end{proof}

\section{Proof of Theorem \ref{th:dimM}}\label{sec3}
Here and throughout, we define a measure $\Pp_{\lambda_d}$
by
\begin{equation}
    \Pp_{\lambda_d}(\cdot) := \int_{\R^d} \Pp^x (\cdot)\, \d x.
\end{equation}
It is easy to see that $\Pp_{\lambda_d}$ is a $\sigma$-finite
measure on the underlying measurable space
$(\Omega\,,\mathcal{F})$.
The corresponding expectation operator will be denoted by
$\E_{\lambda_d}$; that is,
$\E_{\lambda_d}(Z):=\int_{\R^d}\E^x(Z)\,\d x$
for every nonnegative measurable random variable $Z$.

Let $\{\mathcal{F}_t\}_{t\ge 0}$ denote the filtration generated
by $X$, augmented in the usual way.
In order to prove \eqref{eq:dimM}, we
make use of the following strong Markov property
for $\Pp_{\lambda_d}$.

\begin{lemma}\label{Lem:SMP}
    If $f: \R^d \to \R$ is a bounded measurable function and
    $T$ is a stopping time such that $\Pp\{T < \infty\} = 1$, then
    \begin{equation} \label{Eq:MP}
        \E_{\lambda_d} \left[\left.
        f(X(t))\, \right|\, \mathcal{F}_{T}\right]=
        \E^{X(T)} \left[f(X(t-T))\right],
	\end{equation}
	$\Pp_{\lambda_d}$-a.s.\ on $\{T<t\}$.
\end{lemma}

\begin{proof}
    This is well known, particularly for Brownian motion;
    see for example Chung \cite[p.\ 58, Theorem 3]{Chung}.
    We include an elementary self-contained proof.

    If $T$ is nonrandom, say $T=s$ a.s., then (\ref{Eq:MP}) follows
    directly from Proposition 3.2 in \cite{KXZ}. It can be verified
    by elementary computations that
    \eqref{Eq:MP} holds also when $T$ is a discrete stopping time.

    In general, there exists a sequence of discrete stopping times
    $T_n$ such that $T_n \downarrow T$. For any event $A \in
    \mathcal{F}_{T}$, we have $A \in \mathcal{F}_{T_n}$. It follows
    that for all bounded and continuous functions $f$ and $g$ on
    $\R^d$ ($g \in L^1(\R^d)$),
    \begin{equation} \label{Eq:MP2}
        \E_{\lambda_d} \big[f(X(t)) g(X(T_n)) \1_{A \cap \{T_n < t\}} \big]\\
        = \E_{\lambda_d} \left[\E^{X(T_n)} \left(f(X(t-T_n))\right) g(X(T_n))
            \1_{A \cap \{T_n < t\}} \right].
    \end{equation}
    Since the function $\E^x[f(X(s))]$ is continuous in the variables $(x\,, s)$,
    the integrand in the last expression tends to $\E^{X(T)}[f(X(t-T))]
    g(X(T)) \1_{A \cap \{T < t\}}$ almost surely as $n \to \infty$. Hence we can
    apply the dominated convergence theorem to derive \eqref{Eq:MP}
    from \eqref{Eq:MP2}.
\end{proof}

\begin{proof}[of Theorem \ref{th:dimM}: \eqref{eq:dimM}]
    Given $\mu\in\mathcal{P}(F)$ and $\e>0$, let us define
    \begin{equation}
        \ell_{\e,\mu} := \int \frac{\1_{B(0,\e)}(X(s))}{(2\e)^d}\,\mu(\d s).
    \end{equation}
    Note that $\E_{\lambda_d}(\ell_{\e,\mu})=1$.

	Now, let $T:=T_F(\e):=\inf\{ t\in F:\, |X(t)|\le \e\}$. Then
	$T$ is a stopping time and by Lemma \ref{Lem:SMP},  for all $n\ge 1$,
	\begin{equation}\begin{split}
	\E_{\lambda_d} \left( \ell_{2\e,\mu}\, \big|\, \mathcal{F}_{T\wedge n}\right)
		&\ge \frac{1}{(4\e)^d}
		\int_{T\wedge n}^\infty \Pp_{\lambda_d} \left( |X(s)|\le 2\e\,
		\big|\, \mathcal{F}_{T\wedge n} \right) \, \mu(\d s)\\[\bigskipamount]
	&\ge \frac{1}{(4\e)^d}
		\int_T^\infty \kappa_\e(s-T) \, \mu(\d s)\cdot
		\1_{\{T<n\}}.
	\end{split}\end{equation}
	We have used the triangle inequality, together with the fact
	that $|X(T\wedge n)|\le\e$, $\Pp_{\lambda_d}$-a.s.\ on $\{T<n\}$.
	Since $\E_{\lambda_d}\big(\ell_{2\e,\mu}\big) = 1$, we find that
	\begin{equation}\begin{split}
		1 &= \E_{\lambda_d}\left[\E_{\lambda_d} \left(
			\left. \ell_{2\e,\mu}\, \right|\,
			\mathcal{F}_{T\wedge n}\right) \right]\\
		&\geq \frac{1}{(4\e)^d}\,
			\E_{\lambda_d} \left[\int_T^\infty \kappa_\e(s-T) \, \mu(\d s)\cdot
			\1_{\{T<n\}}\right].
    \end{split}\end{equation}
    Denote by $\mu_\infty(\d t):= \Pp_{\lambda_d}(T\in \d t \,|\, T< \infty)$.
    We can let $n\to\infty$ and deduce the following from the monotone
    convergence theorem:
    \begin{equation}\begin{split}
        1 &\geq \frac{1}{(4\e)^d}\, \E_{\lambda_d}
        	\left[\int_T^\infty \kappa_\e(s-T) \,
		\mu(\d s)\cdot \1_{\{T<\infty\}}\right].\\
        &= \frac{1}{(4\e)^d} \int_0^\infty \mu_\infty(\d t)
        	\int_t^\infty \mu(\d s)\, \kappa_\e(s-t) \cdot \Pp_{\lambda_d}\{T<\infty\}.
    \end{split}\end{equation}
    This holds for all probability measures $\mu$. We now
    choose $\mu:=\mu_\infty$ to find that for all $\e>0$,
    \begin{equation}\label{hit:UB}
        \Pp_{\lambda_d} \left\{ \inf_{t\in F} |X(t)|\le\e\right\}
        \le \frac{2^{d+1}(2\e)^d}{\inf_{\nu\in\mathcal{P}(F)}
        \int\int \kappa_\e(|t-s|)
        \,\nu(\d s)\,\nu(\d t)}.
    \end{equation}

    On the other hand, by the Fubini--Tonelli theorem,
    \begin{equation}\label{Eq:Kol}\begin{split}
        \Pp_{\lambda_d}\left\{ \inf_{t\in F}|X(t)|\le \e\right\}
        &= \int_{\R^d} \Pp \Big\{ \inf_{t\in F}|x+X(t)|\le \e\Big\}\, \d x\\
        &=\E \left\{ \lambda_d \left[\left( (X(F) \right)^\e\right] \right\},
    \end{split}\end{equation}
    where $G^\e :=\overline{G+B(0\,,\e)}$ denotes the closed $\e$-enlargement of
    $G$ in the $\ell^\infty$-metric of $\R^d$.

    Let $K_G$ denote the Kolmogorov capacity
    of the set $G$. That is, $K_G(r)$ is the maximal number
    $m$ of points $x_1,\ldots, x_m$ in $G$ with
    $\min_{j\neq k}|x_j-x_k|\geq r$.

    \noindent
    The following three observations will be important for our argument:
    \begin{itemize}
        \item[(i)] $r^d K_G(r)\le \lambda_d(G^r)$ for every $r>0$, where
            $G^r$ denotes the closed $r$-enlargement of $G$.
            Indeed, we can find $k:= K_G(r)$ points $x_1,\ldots,x_k\in G$
            such that $B(x_1\,,r/2),\ldots,B(x_k\,,r/2)$ are disjoint. Since
            $G^r$ contains $\cup_{j=1}^k B(x_j\,,r/2)$ as a subset, the claim
            follows from the monotonicity of the Lebesgue measure.
        \item[(ii)] For every bounded set $G \subseteq \R$, the [upper]
            Minkowski dimension of $G$ is defined by
            $\overline{\dim}_{_{\rm M}} G
            :=\varlimsup_{r\downarrow 0}\log K_G(r)/\log(1/r)$.
            And Tricot \cite{Tricot} (see also Falconer \cite{Fal90})
            has proved that the packing dimension can be defined
            by regularizing $\overline{\dim}_{_{\rm M}}$. Namely, for any set
            $F\subseteq \R$,
            \begin{equation}\label{Eq:Tricot}
               \dimp F  = \inf \sup_{n\ge 1}\, \overline{\dim}_{_{\rm M}} F_n,
            \end{equation}
            where the infimum is take over all bounded Borel sets $F_1,F_2,\ldots$
            such that $F\subseteq\cup_{n=1}^\infty F_n$.
      \item[(iii)]  For every analytic set $G \subseteq \R$,
        \begin{equation}\label{Eq:BD3}
            \overline{\dim}_{_{\rm M}} G\\
            = \sup \left\{\eta> 0:  \varliminf_{\e\downarrow 0}
                \frac{1}{\e^\eta} \inf_{\nu\in\mathcal{P} (G)}
                \int\int \1_{\{|t-s|\le \e\}}\, \nu(\d s)\,\nu(\d t)   =0 \right\}.
        \end{equation}
        This follows from \cite[Theorem 4.1]{KX:07a} which extended in turn an
        earlier result of Howroyd \cite{Howroyd}.
    \end{itemize}

    Consequently, we use \eqref{hit:UB} and \eqref{Eq:Kol} to obtain
    \begin{equation}\label{Eq:Mom}
        \begin{split}
            \E\left[ K_{X(F)}(\e) \right]&\le \frac{1}{\e^d}
                \E\left[\lambda_d \left(\left( X(F)\right)^\e \right) \right]\\
            &\le \frac{2^{2d+1}}{\inf_{\nu\in\mathcal{P}(F)} \int\int\kappa_\e(|s-t|)
                \,\nu(\d s)\,\nu(\d t)}.
        \end{split}
    \end{equation}
    Fix a number $s>\BDim_\kappa F$. By \eqref{Eq:BD},  there
    exists a finite constant $c>0$ such that
    \begin{equation}\label{Eq:BD2}
        \inf_{\nu\in\mathcal{P}(F)}
        \int\int \kappa_\e(|t-s|) \,\nu(\d s)\,\nu(\d t)\ge c \e^s,
    \end{equation}
    for all sufficiently small $\e>0$.
    It follows from (\ref{Eq:Mom}) and (\ref{Eq:BD2})
    that for all $q\in(0\,,1)$
    \begin{equation}
        \Pp \left\{ K_{X(F)}(\e) \ge \e^{-(q+s)} \right\}=O(\e^q)
        \qquad(\e\downarrow 0).
    \end{equation}
    We apply the preceding with $\e:=2^{-n}$,
     and use the Borel--Cantelli lemma
    together with a standard monotonicity, in order to obtain
    the following:
    \begin{equation}
        K_{X(F)}(\e)=O\left( \e^{-(q+s)}\right)\qquad\text{$(\epsilon\downarrow 0)$
        \qquad a.s.}
    \end{equation}
     This proves that $\dimM X(F)\le s+q$ a.s.  Now we first let
    $q\downarrow 0$ and then $s\downarrow \BDim_\kappa F$
    (along countable sequences) to deduce
     the almost sure inequality $\dimM X(F) \le \BDim_\kappa F$.

    Next we prove $\dimM X(F) \ge \BDim_\kappa F$ a.s.
	The definition  \eqref{Eq:BD}
	of $\BDim_\kappa$ implies that for all  $\eta<\BDim_\kappa F$ there exist
    a sequence of positive numbers $\e_n \downarrow 0$ and a sequence
    of measures $\nu_1,\nu_2,\ldots\in\mathcal{P}(F)$ such that
    \begin{equation}\label{Eq:seq}
        \varliminf_{n \to 0}\int\int  \frac{\kappa_{\e_n}(|t-s|)}{\e_n^\eta} \,
        \nu_n(\d s)\, \nu_n(\d t) =0.
    \end{equation}
    If $m_n:=\nu_n\circ X^{-1}$, then $m_n\in\mathcal{P}(X(F))$
    almost surely and
    \begin{equation}
        \E\left[
            \int\int \frac{\1_{\{ |x-y|\le \e_n\}}}{\e_n^\eta}\, m_n(\d y)\, m_n(\d x)
            \right]\\
        =\int\int  \frac{\kappa_{\e_n}(|t-s|)}{\e_n^\eta}
        	\, \nu_n(\d s)\, \nu_n(\d t).
    \end{equation}
    This, Fatou's lemma, and \eqref{Eq:seq} together imply that
    \begin{equation}
        \varliminf_{\e\downarrow 0}\inf_{m\in\mathcal{P}(X(F))}
        \int\int  \frac{\1_{\{ |x-y|\le \e\}}}{\e^\eta}  \, m(\d y)\, m(\d x)
        =0 \quad\text{a.s.}
    \end{equation}
    Consequently, it follows from \eqref{Eq:BD3} that
    $\dimM X(F)\ge\eta$ a.s. Let $\eta$ tend upward to
    $\BDim_\kappa F$ in order to conclude that
    $\dimM X(F) \ge \BDim_\kappa F$ a.s., whence \eqref{eq:dimM}.
\end{proof}

\begin{proof}[of Theorem \ref{th:dimM}: \eqref{eq:dimp}]
    First we prove the upper bound in \eqref{eq:dimp}.
    By the definition  \eqref{Def:Dimka}
    of $\Dim_\kappa$ for all $\gamma >
    \Dim_\kappa F$ there exists a sequence
     $\{F_n\}_{n\geq 1}$ %
    of bounded Borel sets such that
    \begin{equation}\label{Eq:Dim2}
        F \subseteq \bigcup_{n=1}^\infty F_n \quad \hbox{ and }
        \quad \sup_{n\ge 1} \BDim_\kappa F_n < \gamma.
    \end{equation}
    Since $X(F) \subseteq \cup_{n=1}^\infty X(F_n)$, \eqref{Eq:Tricot}
    and Theorem \ref{th:dimM} together imply that
    \begin{equation}\label{Eq:Dim3}
        \dimp X(F) \le \sup_{n\ge 1}\dimM X(F_n) =
        \sup_{n\ge 1} \BDim_\kappa F_n < \gamma \quad \text{ a.s.}
    \end{equation}
    Thus, $\dimp X(F) \le \Dim_\kappa F$ a.s.

    Next we complete the proof of \eqref{eq:dimp}
    by deriving the complementary lower bound,
    \begin{equation}\label{eq:LastGoal}
        \dimp X(F) \ge \Dim_\kappa F
        \qquad\text{a.s.}
    \end{equation}
    It suffices to consider
    only the case that $\Dim_\kappa F>0$; otherwise, there is nothing to prove.

    First we claim that \eqref{Def:Dimka} implies that for
    every $0 < \gamma < \Dim_\kappa F$ there exists a compact subset $E \subseteq F$ such
    that $ \BDim_\kappa (E\cap (s\,, t)) \ge \gamma$ for all $s, t \in
    \mathbf{Q}_+$ and $s < t$ that satisfy $E \cap (s\,, t)\ne \varnothing$. In
    order to verify this claim let us notice that
    if, in addition, $F$ were closed then we could apply the $\sigma$-stability
    of $\Dim_\kappa$ in order to construct a compact
    set $E \subseteq F$ with the desired property as in the proof of
    Lemma 4.1 in Talagrand and Xiao \cite{TX}. In the general case, we proceed
    as in Howroyd's proof of his Theorem 22 \cite{Howroyd}.  Since this is a
    lengthy calculation and not essential to the rest of the proof, we omit
    the details.

    We now demonstrate the a.s.\ lower bound, $\dimp X(E) \ge \gamma$. Since
    $\overline{X(E)}$ and $X(E)$ only differ by at most a countable set, it
    is sufficient to prove $\dimp \overline{X(E)}\ge \gamma$ almost surely.
    Observe that there are at most countably many points in $\overline{X(E)}$
    with the following property: each of them corresponds to a $t \in E$ such
    that $X$ has a jump at $t$, and $t$ cannot be approached from the right
    by the elements of $E$. Removing these isolated points from $\overline{X(E)}$
    yields a closed subset with the same packing dimension as $\dimp \overline{X(E)}$.
    Therefore, without loss of generality, we may and will assume that $\overline{X(E)}$
    is a.s.\ a closed set and every point in $\overline{X(E)}$ is
    the limit of a sequence $X(t_n)$ with $t_n \in E$.

    Since $\overline{X(E)}$ is a.s.\ closed, one can apply Baire's category theorem
    as in Tricot's proof of \cite[Proposition 3]{Tricot}. Thus it suffices to
    prove that almost surely
    \begin{equation}\label{Redu}
        \dimM \left[ \overline{X(E)}\cap B(a\,,r)\right]
        \ge \gamma\quad     \text{ for
        all $a \in \mathbf{Q}^d$ and $r \in \mathbf{Q}\cap (0, \infty)$},
    \end{equation}
    whenever $\overline{X(E)}\cap B(a\,,r) \ne \varnothing$.

    According to the already established first assertion \eqref{eq:dimM} of
    Theorem \ref{th:dimM},
    \begin{equation}\label{Eq:Unidim}
        \Pp\Big\{\dimM  X(E \cap (s, t)) = \BDim_\kappa
        \left[ E\cap (s\,,t) \right] \ \text{ for all }
        s < t \in \mathbf{Q}_+\Big\} = 1.
    \end{equation}
    Fix $a\in\mathbf{Q}^d$ and $r\in\mathbf{Q}\cap(0\,,\infty)$.
    It follows from \eqref{Eq:Unidim} that
    \begin{equation}\label{RHS}
        \dimM \left[ X(E)\cap B(a\,,r) \right]
        \ge \sup \BDim_\kappa \left[ E\cap (s\,,t) \right] \quad \text{a.s.},
    \end{equation}
    where the supremum is taken over all rationals $s,t>0$
    such that $X(E\cap (s\,,t))\subseteq B(a\,,r)$. By the aforementioned assumption
    on $\overline{X(E)}$ we see that, if $X(E) \cap B(a\,, r) \ne \varnothing$,
    then we can always find rationals $s, t >0$ such that $E\cap (s\,,t) \ne \varnothing$ and
    $X(E\cap (s\,,t))\subseteq B(a\,,r)$. This, together with \eqref{RHS}, implies that a.s.\
    $\dimM \left[ X(E)\cap B(a\,,r) \right] \ge \gamma$
    provided $\overline{X(E)}\cap B(a\,,r) \ne \varnothing$.

    Finally, we can choose a $\Pp$-null event such that the preceding holds, off that null event,
    simultaneously for all $a\in\mathbf{Q}^d$ and $r\in\mathbf{Q}\cap(0\,,\infty)$.
    This proves \eqref{Redu}, whence  $\dimp X(E)\ge\gamma$ a.s.;
    \eqref{eq:LastGoal} follows immediately.
\end{proof}

\affiliationone{%
Davar Khoshnevisan\\
Department of Mathematics\\
The University of Utah\\
155 South 1400 East, JWB 233\\
Salt Lake City, Utah 84105--0090\\
USA
\email{davar@math.utah.edu}}
\affiliationtwo{%
Ren\'e L.\ Schilling\\
Institut f\"ur Mathematische Stochastik\\
TU Dresden\\
D-01062 Dresden\\
Germany
\email{rene.schilling@tu-dresden.de}}
\affiliationthree{%
Yimin Xiao\\
Department of Statistics and Probability\\
A-413 Wells Hall\\
Michigan State University\\
East Lansing, MI 48824\\
USA
\email{xiao@stt.msu.edu}}
\end{document}